\documentclass[a4paper,12pt]{amsart}
\usepackage{graphicx}
\usepackage{subfig}
\usepackage[]{amsmath,amssymb,amsthm}

\usepackage[usenames,dvipsnames]{color} 

\newcommand{\R}{\mathbb{R}}

\newtheorem{theorem}{Theorem}[section]
\newtheorem{proposition}[theorem]{Proposition}
\newtheorem{lemma}[theorem]{Lemma}

\theoremstyle{definition}
\newtheorem{definition}[theorem]{Definition}
\theoremstyle{remark}

\usepackage{soul}

\begin{document}

\title[Stability of cycles in a Jungle Game]{Stability of cycles and survival in a Jungle Game with four species}

\author[S.B.S.D. Castro]{Sofia B.S.D. Castro}
\address{S.B.S.D. Castro\\Centro de Matem\'atica da Universidade do Porto\\ Rua do Campo Alegre 687, 4169-007 Porto, Portugal
and
Faculdade de Economia do Porto \\ Rua Dr. Roberto Frias, 4200-464 Porto, Portugal}
\email{sdcastro@fep.up.pt}
\thanks{
The authors are grateful to Alexander Lohse for his comments on a previous version and to David Groothuizen-Dijkema for helping with the code of the numerical simulations and for useful conversations. The insights of two anonymous referees are also gratefully acknowledged.\\
All authors were partially supported by CMUP, member of LASI, which is financed by national funds through FCT -- Funda\c{c}\~ao para a Ci\^encia e a Tecnologia, I.P., under the projects UIDB/00144/2020 and UIDP/00144/2020.
}

\author[A.M.J. Ferreira]{Ana M.J. Ferreira}
\address{A.M.J. Ferreira\\Faculdade de Ci\^encias, Universidade do Porto, Rua do Campo Alegre 687, 4169-007 Porto, Portugal and\\Centro de Matem\'atica da Universidade do Porto\\ Rua do Campo Alegre 687, 4169-007 Porto, Portugal} 
\email{up200800262@edu.fep.up.pt}
\thanks{The second author is the recipient of the PhD grant number PD/BD/150534/2019 awarded by FCT - Funda\c{c}\~ao para a Ci\^encia e a Tecnologia which is co-financed by the Portuguese state budget, through the Ministry for Science, Technology and Higher Education (MCTES) and by the European Social Fund (FSE), through Programa Operacional Regional do Norte.
}
\author[I.S. Labouriau]{Isabel S. Labouriau}
\address{I.S. Labouriau\\Centro de Matem\'atica da Universidade do Porto\\ Rua do Campo Alegre 687, 4169-007 Porto, Portugal}
\email{islabour@fc.up.pt}

\begin{abstract}
{The Jungle Game is used in population dynamics to describe cyclic competition among species that interact via a food chain.
The dynamics of the Jungle Game supports a heteroclinic network whose cycles represent coexisting species. The stability of all heteroclinic cycles in the network for the Jungle Game with four species determines that only three species coexist in the long-run, interacting under cyclic dominance as a Rock-Paper-Scissors Game. 
This is in stark contrast with other interactions involving four species, such as cyclic interaction and intraguild predation.

We use the Jungle Game with four species to determine the success of a fourth species invading a population of Rock-Paper-Scissors players.
}
\end{abstract}

\maketitle


\vspace{.3cm}


\section{Introduction}

Cyclic dominance is observed in competing species when these take turns in increasing and decreasing their numbers.
The mathematical model for interacting species is a set of differential equations.
 The mathematical analog for cyclic dominance,
 as described for example by Hofbauer and Sigmund \cite{HofbauerSigmund}, 
 is a heteroclinic cycle, a union of finitely many equilibria and trajectories that connect them in a cyclic way. Heteroclinic cycles may stand alone or as part of a heteroclinic network, a finite connected union of heteroclinic cycles. In the context of population dynamics, the equilibria typically correspond to the survival of only one species and the connecting trajectories indicate that one species wins over the other, the direction of movement occurring towards the equilibrium representing the winning species whose numbers grow either by feeding on the other or by successfully competing with it. The classic example for cyclic dominance in population dynamics is that of the Uta Stansburiana lizards that play a game of Rock-Scissors-Paper (RSP) to exhibit cyclically increasing numbers of each variety. See, for instance, the work of Allesina and Levine \cite{Allesina_Levine_2011}, May and Leonard \cite{May_Leonard_1975}, Sinervo and Lively \cite{Sinervo_Lively} and Szolnoki \emph{et al} \cite{Mobilia_et_al}. 

The RSP game has been extended to more species. 

One way of doing this preserves the cyclic dominance, as the examples for four species studied by Durney {\em et al.} \cite{Durney2012}, Hua \cite{Hua2013}, Intoy and Pleimling \cite{IntPle2013}, Roman {\em et al.} \cite{Roman2012}, and Szab\'o and Sznaider \cite{SzaSzn2004}, as well as in the review by Dobramysl {\em et al.} \cite{Dobramysl2018}. See Figure~\ref{fig:cyclic-4}. In this case, denoting by $S_i$ the species $i$ and assuming that $S_i$ preys on $S_{i+1}$, $i \mod 4$, there is the formation of alliances of neutral species. Neutral species do not interact with each other and, in this case, the alliances are $S_1$ and $S_3$ as well as $S_2$ and $S_4$. The results are obtained by numerical simulations of the spatial interactions. 
%{\color{blue}
Recently Park {\em et al.} \cite{Park_etal2023}, have studied the dynamics of a model with eight species where two sets of four species fight for space. The symmetry of the interactions between each two species plays a determinant role in the formation of alliances.
%}

\begin{figure}[ht]
\centering
{\includegraphics[width=.8\textwidth]{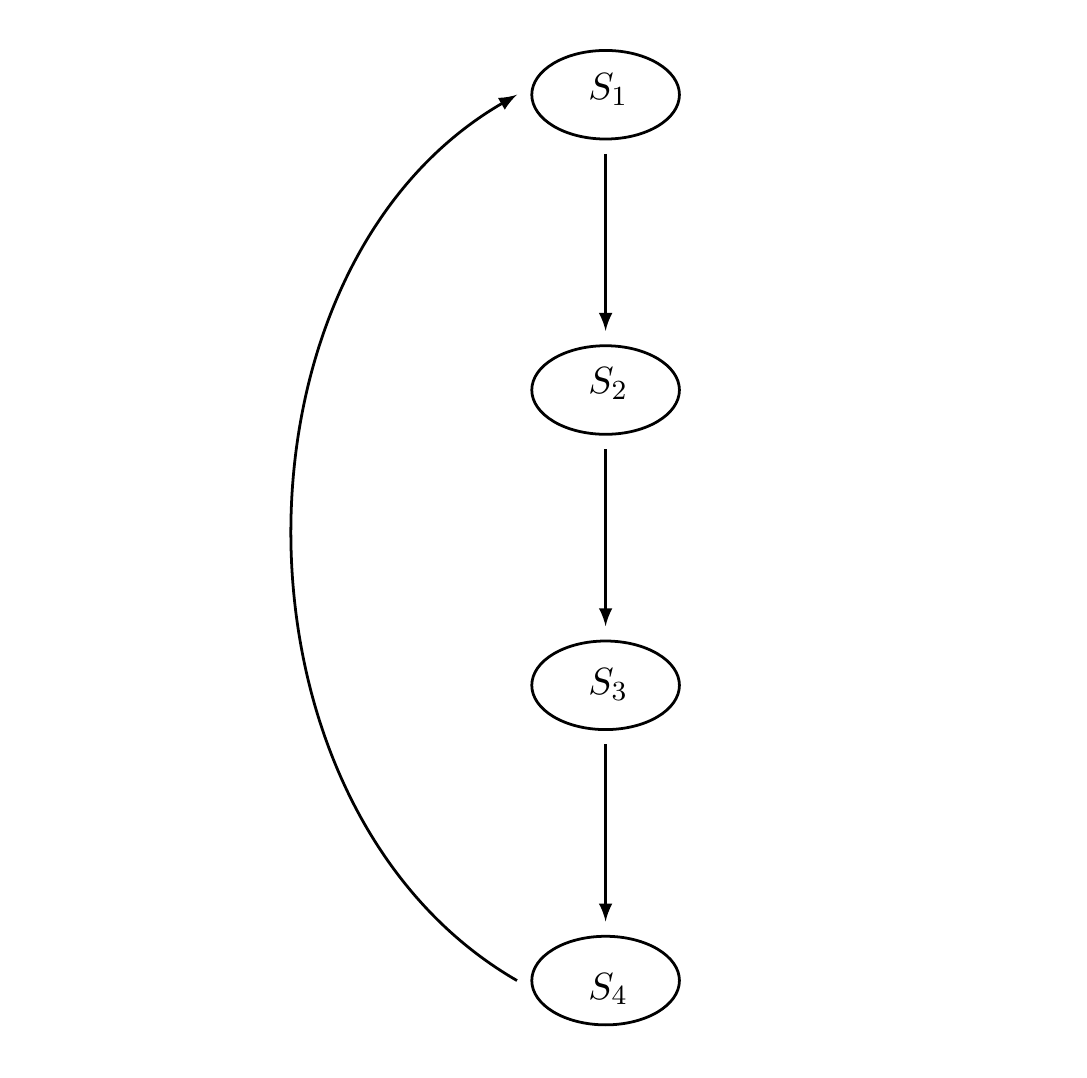}}
\caption{The cyclic relationships of four species: arrows point from winner  
to loser. }
\label{fig:cyclic-4}
\end{figure}

Another way is to include two new species, Lizard and Spock, and adding to the cyclic interaction the property that each species wins over half of the remaining species and loses when confronted with the other half. See the work of Castro {\em et al.} \cite{Castro_Ferreira_Labouriau_Liliana} and Postlethwaite and Rucklidge \cite{Claire_Alastair_2022}.  
%{\color{blue}
The same procedure can be used to obtain a further extension to seven species as in Yang and Park \cite{YanPar2023} who show that in a spatial game\footnote{Further examples of spatial competition involving three, four and five species can be found in the work of Park and co-authors \cite{Park_etal2017}, \cite{ParJan2019}, and \cite{ParJan2021}.}, the level of mobility conditions the number of coexisting species.
%}

In the present article we look at a game that extends RSP by adding a fourth species and constructing a hierarchical interaction to obtain what is called a Jungle Game.
The Jungle Game may be interpreted as describing a hierarchical chain such that each species wins over the ones below, except for the first and last species. The last species wins over the first. This game appears as a chinese board game, also under the name of {\it Dou Shou Qi}, with eight species. See  Kang {\em et al.} \cite{Kang_et_all_2016} 
for a general description of the game and simulations in the case of five species, as well as {the work of Buss and Jackson} \cite{Buss_Jackson_1979}  
for an account of the Jungle Game in coral reefs. As suggested by  \cite{Kang_et_all_2016}
the Jungle Game can also describe the growth and destruction of forests, from the low vegetation state to increasing height and finally destruction by fire.

The interactions present among the four species playing the Jungle Game are depicted in Figure~\ref{fig:4-JG_intro}. These can be precisely described in the following way: let $S_i$, $i=1, \hdots, 4$ denote a species so that the lower the index the closer to the top of the hierarchy the species is. Then $S_1$ preys on $S_2$ and $S_3$, $S_2$ preys on $S_3$ and $S_4$, $S_3$ preys on $S_4$ while $S_4$ preys on $S_1$. The two subsets of three surviving species each, $\{S_1,S_3,S_4\}$ and $\{S_1,S_2,S_4\}$, may coexist in cyclic dominance, that is, playing a RSP game. The survival of a subset of the species also appears as an outcome of the extension of RSP to RSPLS; we do not pursue this similiarity further.

\begin{figure}[ht]
\centering
{\includegraphics[width=.8\textwidth]{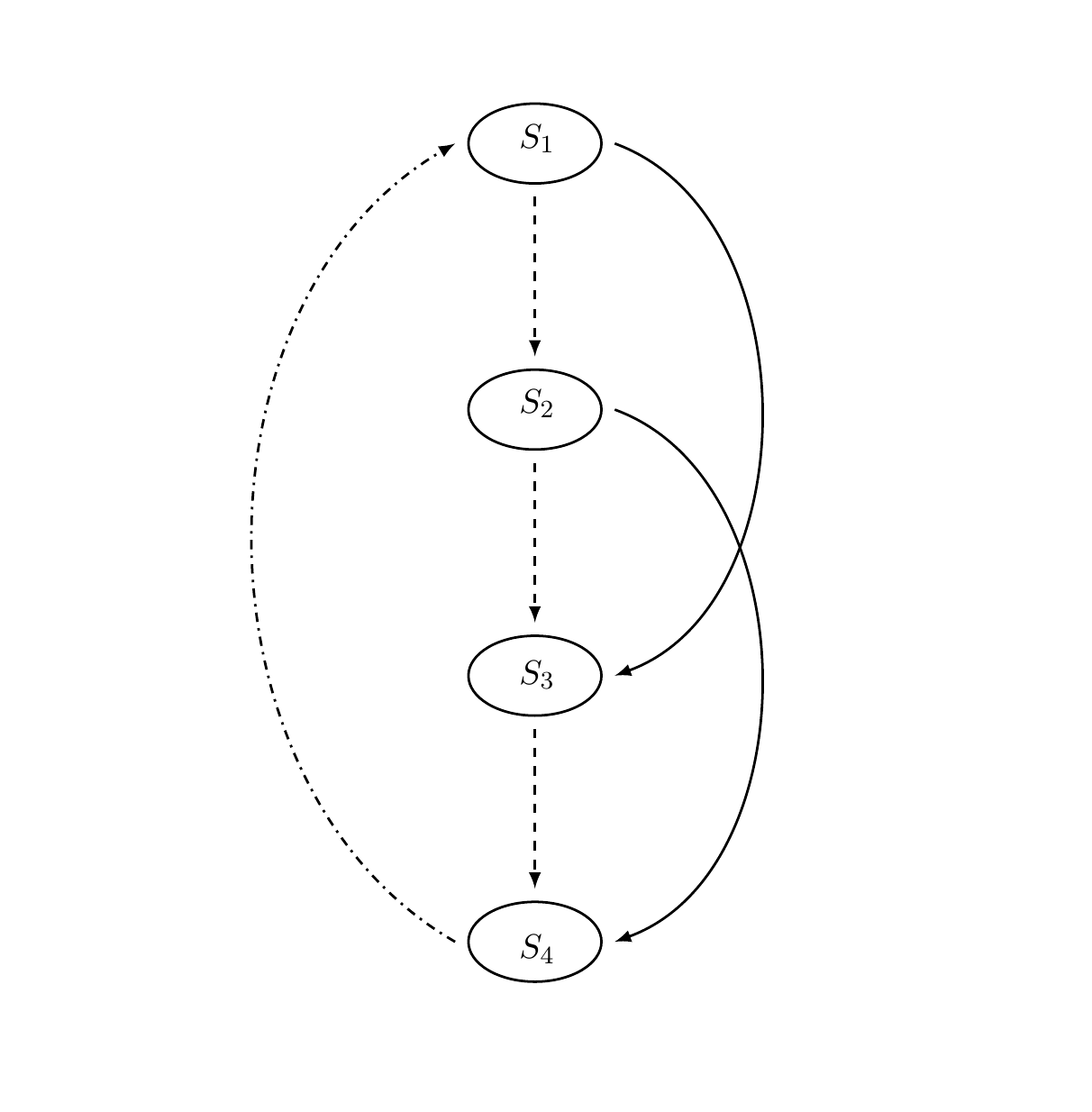}}
\caption{The relationships of four species in the Jungle game; arrows point from winner  
to loser. We distinguish three strengths of interaction: directly from $S_i$ to $S_{i+1}$ by a dashed line, from $S_i$ to $S_{i+2}$ by a solid line, and by a dash-dotted line the interaction between the lowest species $S_4$ and the top species $S_1$.}
\label{fig:4-JG_intro}
\end{figure}

When we represent by $\xi_i$ the equilibrium corresponding to species $S_i$, the connecting trajectories in a heteroclinic cycle point in the opposite direction of those in Figure~\ref{fig:4-JG_intro} since when species $i$ preys on species $j$ the decrease in species $j$ generates an increase in species $i$. The interactions in the Jungle Game with four species lead to a heteroclinic network made of three heteroclinic cycles, one with four equilibria represented by the sequence 
$$
\xi_1 \rightarrow \xi_4 \rightarrow \xi_3 \rightarrow \xi_2 \rightarrow \xi_1,
$$
and two with three equilibria, namely
$$
\xi_1 \rightarrow \xi_4 \rightarrow \xi_2 \rightarrow \xi_1 \;\;\; \mbox{  and   } \;\;\; \xi_1 \rightarrow \xi_4 \rightarrow \xi_3 \rightarrow \xi_1.
$$
 We can associate the stability of each of these cycles to the survival of the species that correspond to its equilibria.
Around a stable cycle there is a large set of initial conditions whose trajectories follow the cycle.
In this way, if the cycle with four equilibria is stable, we expect the coexistence of all four species as time evolves. If one of the cycles with three equilibria is stable we expect one of the species to become extinct while the remaining three species coexist in a RSP game. Which of the species becomes extinct is clear from which heteroclinic cycle is stable: if the stable cycle does not connect the equilibrium $\xi_3$ then the species $S_3$ becomes extinct. Analogously, if the cycle does not connect $\xi_4$, then the species $S_4$ becomes extinct. 

The fact that only these three heteroclinic cycles exist guarantees that there are no other combinations of survival/extinction among the species in the four species Jungle Game. In particular, at most one species becomes extinct and we do not see the formation and survival of neutral alliances as in the strictly cyclic interaction among four species found by the authors of \cite{Durney2012}, \cite{Hua2013}, \cite{IntPle2013}, \cite{Roman2012}, or \cite{SzaSzn2004}.

We depart from the usual spatial or stochastic models and adapt previously established results concerning the stability of heteroclinic networks and their heteroclinic cycles when the model is of competitive Lokta-Volterra interaction. We show that the system of Lotka-Volterra ODEs supports a heteroclinic network consisting of the three heteroclinic cycles above. One of the cycles describes the cyclic dominance among all four species. The other two cycles involve the equilibria corresponding to the two subsets of three species above, describing the cyclic dominance that persists when one, and only one, species becomes extinct and the remaining three interact under a RSP game. Which of these three outcomes is more likely is determined by the stability of the three heteroclinic cycles in the network: only somewhat stable heteroclinic cycles persist in the long-run. We study the stability of each cycle by using the notion of stability index proposed by Podvigina and Ashwin \cite{Podvigina_Ashwin_2011} and results by Lohse \cite{Lohse2015}. We show that only the cycle among the species in the set $\{S_1,S_2,S_4\}$ is stable under some mild assumptions that guarantee that the Jungle Game heteroclinic network is asymptotically stable. The stability of the network as a whole ensures that all initial configurations of the four species converge, in the long-run, to the coexistence of only the three species $S_1$, $S_2$, and $S_4$. We use results from Podvigina {\em et al.} \cite{Podvigina_et_al_2020} to establish the asymtptotic stability of the network.
To the best of our knowledge, we are the first to achieve analitic results concerning extinction while using competitive Lotka-Volterra equations.

A particular instance of the Jungle Game with four species appears in L\"utz {\em et al.} \cite{Lutz2013} where the authors use mean field such that the interaction is of two types only: the cyclic interaction between $S_i$ and $S_{i+1}$ ($i \mod 4$) is stronger than that between $S_i$ and $S_{i+2}$ ($i = 1,2$). From Figure~\ref{fig:4-JG_intro} we see that this is a particular case of ours where the interaction between $S_4$ and $S_1$ is of the same type as that between $S_1$ and $S_2$. The authors of \cite{Lutz2013} illustrate numerically the extinction of the species $S_3$, which we prove in a more general setting. A similar outcome, obtained numerically in a spatial model can be found in L\"utz {\em et al.} \cite{Lutz2017}.

We also interpret the Jungle game as a game describing the introduction of a fourth species, an Alien, into the cyclic dominance of three species modelled by a RSP game.
To the best of our knowledge the literature on models for the interaction of several species addresses the questions: (i) do all species coexists? and (ii) if not, which are the species that become extinct and which those that coexist? Thus our re-interpretation of the Jungle Game as the introduction of the fourth species addresses an issue which appears to be absent from the literature, allowing
 us to also answer a different question: when we have three species coexisting under RSP, how does the appearance of a fourth  species affect survival and extinction? We show that the outcome depends on the nature of the invading species, the Alien. If the Alien is weak, in the sense that it loses against two of the original species, then it becomes extinct and the original species coexist in the long-run. If the Alien is strong, ie.\ it loses against only one of the original species, then one of the original species becomes extinct and is replaced in the RSP by the Alien. We can further determine which species is replaced by the Alien. 

We end this introduction with a note on intraguild predation (IGP). Although IGP is frequently set among four species, the interactioins are not hierarchical. Instead there are two prey that have two types of predators: an IG prey, feeding on the previously mentioned two prey and an IG predator feeding on all three remaining species. See Bl\'e {\em et al.} \cite{Ble2021} and Okuyama \cite{Okuyama2009}. A variation of IGP consists in one prey that is used by two predators and one superpredator that feeds on either all three other species or on only one of the predators. See the work of Wei \cite{Wei2010} for the latter and that of Gakkhar {\em et al.} \cite{Gakkhar2012} and Mondal {\em et al.} \cite{Mondal2022} for the former. These interactions do not produce heteroclinic networks even though the results of \cite{Mondal2022} show the possibility of the coexistence of three of the four species, albeit as one of several other outcomes.

This article proceeds as follows: the next section introduces previous definitions and results in the literature. 
The reader familiar with heteroclinic dynamics may skip (parts of) the next section.
Section~\ref{sec:4-JG} is devoted to the study of the Jungle Game with four species and is divided into four subsections: the first containing results about the asymptotic stability of the network, the second describing the stability of the cycles in the network, the third containing numerical simulations and the fourth describing how the Jungle Game can provide information about the success of an invader. The final section provides some concluding remarks.

\section{Background and notation}\label{sec:background}

Consider a dynamical system described by an ODE
\begin{equation}\label{eq:ODE}
\dot{x} = f(x),
\end{equation}
where $x \in \R^n$ and $f$ is a smooth map from $\R^n$ to itself. For each hyperbolic equilibrium $\xi$ of \eqref{eq:ODE} we denote  its stable and unstable manifolds respectively by $W^s(\xi)$ and $W^u(\xi)$.
Following Ashwin {\em et al.} \cite{AshCasLoh}, given two hyperbolic equilibria of \eqref{eq:ODE}, $\xi_i$ and $\xi_j$, we call $C_{ij} = W^u(\xi_i) \cap W^s(\xi_j),$ a {\em full set of connections} from $\xi_i$ to $\xi_j$. If $\xi_i$ and $\xi_j$ are neither the same equilibrium nor symmetry related, the full set of connections is {\em heteroclinic}. Note that if $\text{dim}(C_{ij})>1$ the connection $C_{ij}$ consists of infinitely many {\em connecting trajectories} $\kappa_{ij}= [\xi_i \rightarrow \xi_j]$, solutions of \eqref{eq:ODE} that converge to $\xi_i$ in backward time and to $\xi_j$ in forward time. 

We are concerned with {\em heteroclinic cycles}, the
union of finitely many hyperbolic saddles, $\xi_1, \hdots , \xi_m$ and connections $C_{j,j+1}$ for $j=1, \hdots , m$ with $\xi_{m+1}=\xi_1$. 
Generically, a connection between two saddles is not robust. However, the existence of flow-invariant spaces where $\xi_{j+1}$ is a sink leads to robust connections of saddle-sink type
between $\xi_j$ and $\xi_{j+1}$.
Such flow-invariant spaces appear naturally in equivariant dynamics, in the form of fixed-point spaces, as well as in game theory and population dynamics, in the form of either coordinate hyperplanes (Lotka-Volterra systems) or hyperfaces of a simplex (replicator dynamics). A connected union of finitely many heteroclinic cycles is a {\em heteroclinic network}.

In the context of heteroclinic networks and cycles the equilibria are sometimes called {\em nodes}.

We focus on the stability of heteroclinic cycles that are part of the same heteroclinic network. In such case there is at least one equilibrium that belongs to more than one cycle. An equilibrium $\xi_i$ such that there are two connecting trajectories $\kappa_{ij}$ and $\kappa_{ik}$ belonging to two different heteroclinic cycles is called a {\em distribution node} in Castro and Garrido-da-Silva \cite{castro2022_switching}. It is clear that, in this case, $W^u(\xi_i)$ is not contained in a single heteroclinic cycle and therefore, cycles in a heteroclinic network are never asymptotically stable. Several intermediate notions of stability exist in the literature. The ones that are relevant for our results are those of \textit{essential asymptotic stability} (e.a.s.) introduced by Melbourne \cite{Melbourne_1991} and refined by Brannath \cite{Brannath1994} and of \textit{complete instability} (c.u., for {\em completely unstable}) as in Krupa and Melbourne \cite{KruMel1995}. A heteroclinic cycle which is e.a.s.\ attracts a set of almost full measure in its neighbourhood whereas a c.u.\ heteroclinic cycle attracts a set of measure zero. Rigorous definitions are given in Subsection \ref{subsec:cycles} below.

The stability of heteroclinic cycles is studied by combining the concept of {\em stability index} of Podvigina and Ashwin \cite{Podvigina_Ashwin_2011} with the results of Lohse \cite{Lohse2015}: a heteroclinic cycle such that all its connections have a positive stability index is e.a.s. The calculation of stability indices is by no means straightforward but those for the Jungle Game with four species can be obtained by results in \cite{Podvigina_Ashwin_2011}. 
We provide more detail in Subsection~\ref{subsec:cycles} below. 
Since existing results apply to 1-dimensional connecting trajectories, when 2-dimensional connections exist as in the Jungle Game, we consider the subcycle consisting of 1-dimensional connecting trajectories inside coordinate planes. Details are given in Section~\ref{sec:4-JG}.
The stability of the heteroclinic network as a whole can be determined by using \cite{Podvigina_et_al_2020} as described in the next subsection.

There is a close relation between stability and survival. Each equilibrium for the dynamics corresponds to the survival of a single species. Because in a heteroclinic cycle no one equilibrium is stable, we do not expect to observe the extinction of all species but one. The species that survive and coexist are those that correspond to equilibria in a stable, either e.a.s. or asymptotically stable, heteroclinic cycle. The heteroclinic connections in the stable heteroclinic cycle determine the interaction among  the surviving species. The species that correspond to equilibria that only belong to unstable heteroclinic cycles
become extinct in the long-run.
Therefore, only species corresponding to heteroclinic cycles with positive stability indices for the connections are expected to survive and coexist.

\subsection{Results for the stability of a network}\label{subsec:back-stab-net}

Podvigina \textit{et al.} \cite{Podvigina_et_al_2020} provide sufficient conditions for the asymptotic stability of ac-networks. These are such that all equilibria are located on coordinate axes (not more than one equilibrium per half-axis), their unstable manifold is entirely contained in the network and the network is invariant under an action of $\mathbb{Z}_2^n$. Their results concerning stability apply to heteroclinic networks without symmetry as well. For ease of reference, we introduce the necessary notation and definitions and transcribe the relevant results below.

Consider a node $\xi_j$ with an incoming connection from a node $\xi_i$ and an outgoing connection towards a node $\xi_k$. We classify the eigenvalues of the Jacobian matrix at $\xi_j$ as in Podvigina~\cite{Podvigina_2012}. Denote the eigenvalues of the Jacobian matrix at $\xi_j$ by $-r_j$ (radial) in the direction of $\xi_j$ to the origin, by $-c_{ji}$ (contracting) in the direction of $\xi_i$, by $e_{jk}$ (expanding) in the direction of $\xi_k$, and by $t_{j,\ell}$, $\ell=1,...,s$ (transverse) the eigenvalues in all other $s$ directions. The radial and contracting eigenvalues are negative, the expanding eigenvalues are positive but the transverse ones can have either sign.

\begin{figure}[ht]
\centering
\includegraphics[width=.8\linewidth]{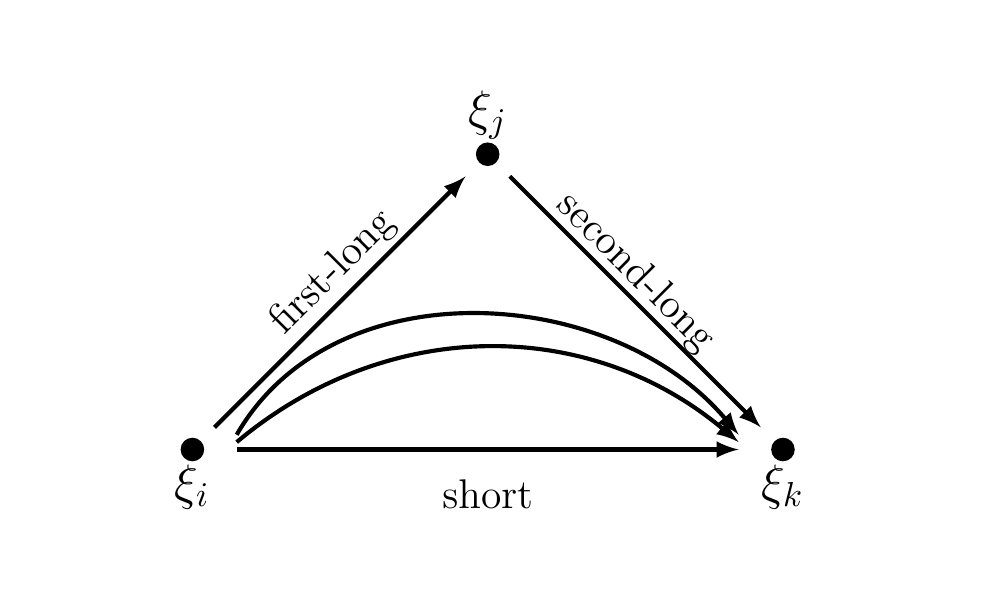}
\caption{A $\Delta$-clique $\Delta_{ijk}$. The short connection is $C_{ik}$, the first-long connection is $C_{ij}$ and the second-long connection is $C_{jk}$.}\label{fig:deltaclique}
\end{figure}

Since a necessary condition for the asymptotic stability of a heteroclinic network is that the unstable manifold of every node be contained in the network, a structure called a $\Delta$-{\em clique} inevitably appears, see Definition 2.1 in \cite{Podvigina_et_al_2020}. Loosely speaking, a $\Delta$-clique $\Delta_{ijk}$, illustrated in Figure~\ref{fig:deltaclique}, is a subset of a heteroclinic network, made of 3 nodes and 3 full sets of connections, that is not a cycle. 
Every $\Delta$-clique, $\Delta_{ijk}$, has a beginning point (b-point), a medium point (m-point) and an end point (e-point), denoted respectively by $\xi_i$, $\xi_j$, and $\xi_k$ in Figure~\ref{fig:deltaclique}.
The short connection of a $\Delta$-clique is 2-dimensional and connects the b- and e-points directly. The first-long and second-long connections are 1-dimensional and together they indicate a transition from the b-point to the e-point through the m-point.

We re-write Lemmas 4.4, 4.6 and 4.8 from \cite{Podvigina_et_al_2020} according to our context, where $c=\max c_{ij}$ and $e=\max e_{ij}$. 
The conditions for asymptotic stability of a network are formulated in Corollary 4.18 of \cite{Podvigina_et_al_2020} in terms of some quantities $\rho_j$ associated to each node in the network that may be computed using the lemmas below. 

\begin{lemma}[\it Lemma 4.4 from \cite{Podvigina_et_al_2020}] \label{lemma1}
Suppose that an equilibrium $\xi_j \in Y$ is not an m-point for any of the $\Delta$-cliques of Y. 
Then
\begin{equation}
\rho_j=min(c/e, 1-t/e)
\end{equation}
\end{lemma}

\begin{lemma}[\it Lemma 4.6 from \cite{Podvigina_et_al_2020}] \label{lemma2}
Suppose that an equilibrium $\xi_j \in Y$ is an m-point for several (one or more) $\Delta$-cliques of $Y$ and it has one expanding eigenvector. 
Then
\begin{equation}
\rho_j=min(c/e, 1)
\end{equation}
\end{lemma}

\begin{lemma}[\it Lemma 4.8 from \cite{Podvigina_et_al_2020}] \label{lemma3}
Suppose that an equilibrium $\xi_j \in Y$ has one contracting eigenvector, two expanding eigenvectors and $\xi_j$ is an m-point for just one $\Delta$-clique. The contracting eigenvector is the f-long vector of the $\Delta$-clique and the expanding eigenvector associated to the eigenvalue $e_2$ is not a s-long vector of the $\Delta$-clique. Then
\begin{equation}
\rho_j=\frac{c}{c+e_2}
\end{equation}
\end{lemma}

Corollary 4.18 from \cite{Podvigina_et_al_2020} states that if the product of the $\rho_j$ for each cycle in the network is bigger than one then the network is asymptotically stable.

\subsection{Results for the stability of cycles}\label{subsec:cycles} 

The stability of the heteroclinic cycles in a network is determined by using the stability index $\sigma$ introduced in \cite{Podvigina_Ashwin_2011} as we proceed to explain.

Consider the flow $\Phi(t,x)$ of a differential equation in $\R^n$. Let $X \subset \R^n$ be a flow-invariant set with compact closure $\overline{X}$ and for $\epsilon>0$ its  $\epsilon$-neighbourhood
$ B_{\epsilon}(X)=\{x \in \R^n: d(x,\overline{X})<\epsilon\}.$
For $\delta>0$, the $\delta$-local basin of attraction of $X$ is:
$$
 \mathcal{B}_{\delta}(X)=\left\{x\in\mathbb{R}^n : \Phi(t,x)\in B_\delta(X)  \text{ for any } t\geq 0 \text{ and } \omega(x)\subset X\right\}.
$$

\begin{definition}
The set $X$ is {\em essentially asymptotically stable} (e.a.s.) if the measure of its $\delta$-local basin of attraction, $\mathcal{B}_{\delta}(X)$, tends to full measure in a $\epsilon$-neighbourhood, $B_{\epsilon}(X)$, of $X$ as $\delta$ and $\epsilon$ become small, that is, if
 \begin{equation*}
  \lim_{\delta \to 0} \left[\lim_{\epsilon \to 0}\frac{\ell\left(B_{\epsilon}(X) \cap  \mathcal{B}_{\delta}(X) \right)}{\ell\left(B_{\epsilon}(X) \right)}\right]=1
 \end{equation*}
where $\ell$ is the Lebesgue measure.
\end{definition}

\begin{definition}
 The set $X$ is {\em completly unstable} (c.u.) if there exists $\delta >0$ such that the $\delta$-local basin of attraction of $X$ is of measure zero, that is, $\ell\left(\mathcal{B}_{\delta}(X)\right)=0$
\end{definition}

The local stability index of a cycle at a point in a connection, defined below, measures how much of its neighbourhood returns and is attracted to it after following the cycle to which the connection belongs. Results in \cite{Podvigina_Ashwin_2011} and Lohse \cite{Lohse2015} relate the stability of a cycle with the sign of the stability indices of every connection in that cycle. If the stability index is positive for every connection then the cycle is e.a.s. and if all the stability indices are equal to $-\infty$ then the cycle is c.u.

\begin{definition}
The {\em local stability} index $\sigma(x)$ of a cycle $X$ at a point $x\in X$ in a connection is the difference $\sigma(x)=\sigma_+(x)-\sigma_-(x)$ where
$$
\sigma_+(x)=\lim_{\delta\to 0}\lim_{\varepsilon\to 0}
\dfrac{\ln(\Sigma_{\varepsilon,\delta}(x)}{\ln\varepsilon}
\qquad
\sigma_-(x)=\lim_{\delta\to 0}\lim_{\varepsilon\to 0}
\dfrac{\ln(1-\Sigma_{\varepsilon,\delta}(x)}{\ln\varepsilon}
$$
where 
$$
\Sigma_{\varepsilon,\delta}(x)=\dfrac{\ell\left(B_\varepsilon(x)\cap \mathcal{B}_{\delta}(X) \right) }{\ell(B_\varepsilon(x))}
$$
\end{definition}

The stability index is calculated in \cite{Podvigina_Ashwin_2011} for simple cycles of types $B$ and $C$ as defined by Krupa and Melbourne \cite{krupa_melbourne_2004}. Suppose that $\Gamma \subset O(n)$ is a finite Lie group acting linearly on $R^n$. A simple robust heteroclinic cycle $X \subset \R^4$ is of type B if there is a fixed-point subspace $Q$ with $\dim Q = 3$ such that $X \subset Q$. 

We transcribe, in the following two lemmas, the results from \cite{Podvigina_Ashwin_2011} that we are going to use. The superscript ``$-$'' indicates that $-I \in \Gamma$ and the subscript ``$3$'' indicates the number of equilibria in the cycle. The function $f^\textnormal{index}$ is used to calculate the local stability indices is defined as $f^\textnormal{index}(\alpha,\beta)=f^+(\alpha,\beta)-f^+(-\alpha,-\beta)$ where
\[
f^{+}(\alpha,\beta)=\begin{cases}
+\infty, & \alpha\geq0,\; \beta\geq0\\[0.1cm]
0, & \alpha\leq0,\; \beta\leq0\\[0.1cm]
-\dfrac{\beta}{\alpha}-1, & \alpha<0,\; \beta>0, \; \dfrac{\beta}{\alpha}<-1\\[0.5cm]
0, & \alpha<0,\; \beta>0, \; \dfrac{\beta}{\alpha}>-1\\[0.5cm]
-\dfrac{\alpha}{\beta}-1, & \alpha>0,\; \beta<0, \; \dfrac{\alpha}{\beta}<-1\\[0.5cm]
0, & \alpha>0,\; \beta<0, \; \dfrac{\alpha}{\beta}>-1.
\end{cases}
\]

For a heteroclinic cycle in $\R^4$ the Jacobian matrix at any equilibrium $\xi_j$ has exactly four eigenvalues, one of each type. Let $a_j=\frac{c_j}{e_j}$ and $b_j=-\frac{t_j}{e_j}$ where $-c_j$ is the contracting eigenvalue, $e_j$ the expanding eigenvalue and $t_j$ the transverse eigenvalue.

\begin{lemma}\label{lemma1:calc-index}
For a cycle of type $B_3^-$ with $b_1<0, b_2>0$ and $b_3>0$, the stability indices along connecting trajectories are as follows: 
\begin{enumerate}
  \item If $a_1a_2a_3<1$ or $b_1a_2a_3+b_3a_2+b_2<0$, then the cycle is not an attractor and the stability indices are $-\infty$.
  \item If $a_1a_2a_3>1$ and $b_1a_2a_3+b_3a_2+b_2>0$, then the stability indices are $\sigma_1=f^\textnormal{index}(b_1,1) $, $\sigma_2=+\infty$ and $\sigma_3=f^\textnormal{index}(b_3+b_1a_3,1)$.
\end{enumerate}
\end{lemma}

\begin{lemma}\label{lemma2:calc-index}
 For a cycle of type $B_3^-$ with $b_1<0, b_2<0$ and $b_3>0$, the stability indices along connecting trajectories are as follows:
\begin{enumerate}
  \item If $a_1a_2a_3<1$ or $b_2a_1a_3+b_1a_3+b_3<0$ or $b_1a_2a_3+a_2b_3+b_2<0$, then the cycle is not an attractor and the stability indices are $-\infty$.
  \item If $a_1a_2a_3>1$ and $b_2a_1a_3+b_1a_3+b_3>0$ and $b_1a_2a_3+a_2b_3+b_2>0$, then the stability indices are $\sigma_1=\min(f^\textnormal{index}(b_1,1), f^\textnormal{index}(b_1+b_2a_1,1))$, $\sigma_2=f^\textnormal{index}(b_2,1)$ and $\sigma_3=f^\textnormal{index}(b_3+b_1a_3,1)$.
\end{enumerate}
\end{lemma}

The stability indices above are of a single connecting trajectory.

\section{The Jungle Game with four species}\label{sec:4-JG}

The Jungle Game with four species describes the hierarchical interaction among four species, $S_1$, $S_2$, $S_3$, and $S_4$. We assume $S_1$ is at the top and $S_4$ at the bottom of the hierarchy so that  $S_1$ wins over $S_2$ and  $S_3$, $S_2$ wins over $S_3$ and $S_4$, $S_3$ wins over $S_4$, and $S_4$ wins over $S_1$. This hierarchy leads to an increase in the numbers of the winning species as a result of a decrease of the losing species. The relations among the species and the graph of the corresponding heteroclinic network appear in Figure \ref{fig:4-species}.

\begin{figure}[ht]
\centering
{\includegraphics{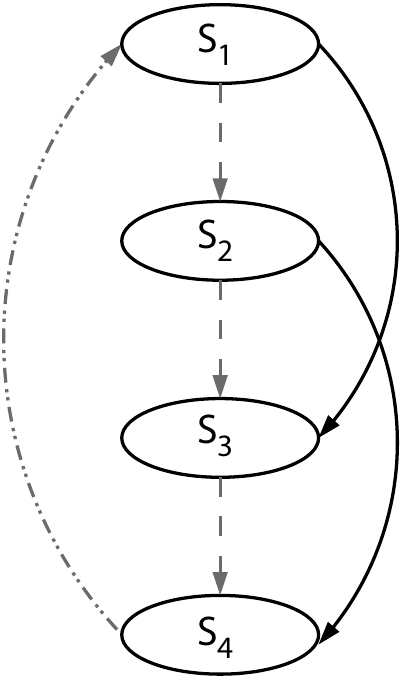}}
\hspace{1cm}
{\includegraphics{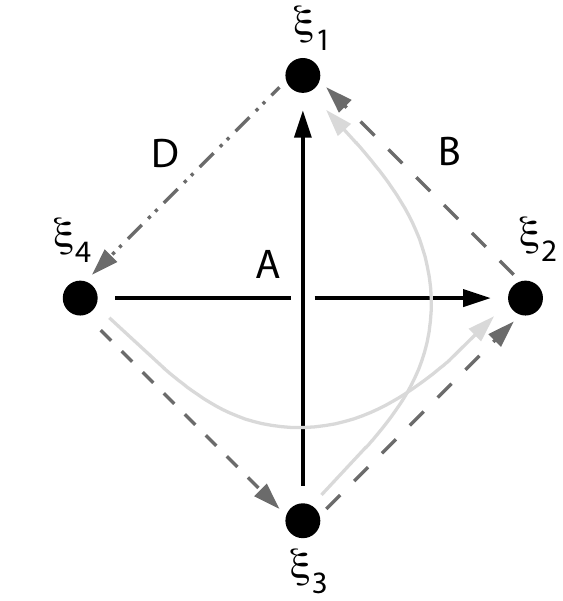}} \\
(a) \hspace{5.5cm} (b)
\caption{(a) The relationships of four species in the Jungle game; arrows point from winner  
to loser.
(b) The four species Jungle game network: 2-dimensional connections are indicated by the letter A and solid lines, 1-dimensional connections in a $\Delta$-clique by B and dashed lines,  and other 1-dimensional connections by D and dash-dotted lines. 
The grey arrows signal the existence of a $\Delta$-clique.
There are two $\Delta$-cliques in the heteroclinic network of the Jungle game with four species: $\Delta_{321}$ and $\Delta_{432}$.}
\label{fig:4-species}
\end{figure}

Modelling the interactions as Lotka-Volterra, we obtain the following system of ODEs
\begin{equation}\label{eq:4-species}
 \begin{cases}
    \Dot{x}_1=x_1 \left(1- R+e_{21}x_2 + e_{31}x_3 - c_{41}x_4\right) \\
    \Dot{x}_2=x_2 \left(1- R-c_{12}x_1 + e_{32}x_3 + e_{42}x_4\right) \\
    \Dot{x}_3=x_3 \left(1- R-c_{13}x_1 - c_{23}x_2 + e_{43}x_4\right) \\
    \Dot{x}_4=x_4 \left(1-R+e_{14}x_1 - c_{24}x_2 - c_{34}x_3\right)
 \end{cases}
\end{equation}
where $R=x_1 + x_2 + x_3 + x_4$ and $x_i \in \R_+$ denotes individuals of the species $S_i$. The equilibrium corresponding to the survival of only species $S_j$ is denoted by $\xi_j$.
We label the coefficients in such a way that the indices $ij$ denote an eigenvalue of the Jacobian matrix at $\xi_i$ in the direction of $x_j$.

\begin{theorem}\label{th:inv-sphere}
If $e_{ij}<1$ for all $i,j$ then there exists an attracting invariant sphere for the dynamics of \eqref{eq:4-species}. Furthermore, system \eqref{eq:4-species} supports the heteroclinic network of Figure~\ref{fig:4-species} (b).
\end{theorem}

\begin{proof}
We use Theorem 5.1 in Field~\cite{Field1989} to prove the existence of an attracting invariant sphere. In order to verify the hypotheses we change variables to obtain a vector field whose nonlinear part is homogeneous of degree 3. It is easy to verify that defining $X_i^2 = x_i$ with $X=(X_1,\ldots,X_4)$, and denoting by $Q(X)$ their non-linear part, the ODEs in the new variables $X_i$ satisfy the hypotheses of this theorem provided $e_{ij}<1$ for all $i,j$. In fact, using the standard inner product, $\langle ., . \rangle$, we obtain
\begin{eqnarray*}
\langle Q(X), X \rangle & = & \frac{1}{2} X_1^2 \left( -X_1^2 - (1-e_{21})X_2^2 - (1-e_{31})X_3^2 - (1+c_{41})X_4^2\right) + \\
 & + &  \frac{1}{2} X_2^2 \left( -(1+c_{12})X_1^2 - X_2^2 - (1-e_{32})X_3^2 - (1-e_{42})X_4^2\right) + \\
 & + &  \frac{1}{2} X_3^2 \left( -(1+c_{13})X_1^2 - (1+c_{23})X_2^2 - X_3^2 - (1-e_{43})X_4^2\right) + \\
 & + &  \frac{1}{2} X_4^2 \left( -(1-e_{14})X_1^2 - (1+c_{24})X_2^2 - (1+c_{34})X_3^2 - X_4^2\right)
\end{eqnarray*}
It is clear that $\langle Q(X), X \rangle < 0$ if $e_{ij}<1$ for all $i,j$.

That the vector field supports a heteroclinic network, follows from the facts that there are no equilibria outside of the coordinate axes in $\R_+^4$ and that the subspaces $\{x_i=0\}$ are flow-invariant. This can be checked directly from \eqref{eq:4-species}.
\end{proof}

There are three cycles in the heteroclinic network in the dynamics of the Jungle Game with four species: $[\xi_1 \to \xi_4 \to \xi_3 \to \xi_2 \to \xi_1]$, 
$[\xi_1 \to \xi_4 \to \xi_2 \to \xi_1]$ 
and $[\xi_1 \to \xi_4 \to \xi_3 \to \xi_1]$ and two $\Delta$-cliques ($\Delta_{321}$ and $\Delta_{432}$).

We assume that the interaction between any two of the species has the same strength when they occur along edges of similar type.
This allows us to group the eigenvalues of the Jacobian matrix at each equilibrium according to the type of connection, as suggested by Figure~\ref{fig:4-species}(b): species $S_i$ wins over species $S_{i+1}$ along a connection of type B, species $S_i$ wins over species $S_{i+2}$ along a connection of type B, and $S_4$ wins over $S_1$ along a connection of type D. We define
\begin{eqnarray*}
    e_A = e_{31}=e_{42} \qquad & \quad e_B=e_{43}=e_{32}=e_{21} \quad & \qquad e_D=e_{14} \\
    c_A = c_{13}=c_{24} \qquad & \quad c_B=c_{34}=c_{23}=c_{12} \quad & \qquad c_D=c_{41}.
\end{eqnarray*}

As usual, we classify these eigenvalues as radial, contracting, expanding and transverse.
As pointed out in the Introduction, the work of \cite{Lutz2013} presents the mean field model in the particular case when $c_B=c_D$ and $e_D=e_B$.

\subsection{Asymptotic stability of the network}\label{subsec:stab-network}

We prove that, for an open set of parameter values, the heteroclinic network in the Jungle Game with four species is asymptotically stable. This guarantees that all trajectories that begin in a neighbourhood of the network remain close to, and eventually converge to, the network. When combined with the stability of each cycle, this provides a complete description of the long-run dynamics. In a population model this property is essential in order to ensure that the future outcome around the network is only that of the prescribed relationships of the species present.

We use the results described in Subsection~\ref{subsec:back-stab-net} and, without loss of generality, assume further that
\begin{equation}
 \min c > \max e \; , \qquad  e_A > e_B , \qquad \mbox{ and   }   \qquad c_A > c_B.
 \label{eq:assumption}
\end{equation}
The first condition is the weakest usual necessary condition for stability. The second and third conditions determine the tangency of trajectories in a $\Delta$-clique: trajectories move away from a b-point in the direction of the first-long connection and approach a e-point along the second-long connection. This corresponds to the way in which trajectories in a $\Delta$-clique are typically drawn.
Note that the authors in \cite{Podvigina_et_al_2020} are considering as transverse the eigenvalues that are neither radial, contracting nor expanding for any cycle through $\xi_j$. 
These eigenvalues are called global transverse in \cite{castro2022_switching}.
The Jungle Game network does not have 
global  transverse eigenvalues.
\begin{theorem}\label{th:stab-net}
Assume that \eqref{eq:assumption} holds.
The heteroclinic network supported by \eqref{eq:4-species} is asymptotically stable if
\begin{equation}\label{eq:suf-cond-net}
c_B^2c_D>(c_B+e_A)e_B e_D. 
\end{equation}
\end{theorem}

\begin{proof}
The quantity $\rho_j$ is obtained from Lemmas \ref{lemma1}, \ref{lemma2} and \ref{lemma3} for all equilibria in each one of the cycles and depends on whether the node $\xi_j$ is, or is not, an m-point in one or more $\Delta$-cliques of the network. 
Since the Jungle Game network does not have global transverse eigenvalues, for this network  in Lemma \ref{lemma1} we have $\rho_j=c/e$.

From Corollary 4.18 of \cite{Podvigina_et_al_2020}, the Jungle Game network of four species is asymptotically stable if, for any cycle $\Sigma$ of the network, 
$\rho( \Sigma)=\prod \rho_j>1$,
 that is, if the following conditions are satisfied:

\begin{eqnarray}
\rho(\Sigma_{142})=\rho_1 \rho_4 \rho_2 &=& \frac{c_B}{e_D} \cdot \frac{c_D}{e_A} >1 \label{T1} \\ \nonumber \\ 
\rho(\Sigma_{143})=\bar{\rho}_1 \bar{\rho}_4 \bar{\rho}_3 &=& \frac{c_A}{e_D} \cdot \frac{c_D}{e_B} \cdot \frac{c_B}{c_B+e_A} >1  \label{T2} \\ \nonumber \\
\rho(\Sigma_{1432})=\tilde{\rho}_1 \tilde{\rho}_4 \tilde{\rho}_3 \tilde{\rho}_2 &=& \frac{c_B}{e_D} \cdot \frac{c_D}{e_B} \cdot \frac{c_B}{c_B+e_A} >1 \label{T3}
\end{eqnarray}
Equations (\ref{T1}), (\ref{T2}) and (\ref{T3}) are the $\rho(X)$ when $X$ is the cycle $\Sigma_{142}$, $\Sigma_{143}$ and $\Sigma_{1432}$ respectively. Once condition (\ref{eq:assumption}) holds, then condition (\ref{T1}) is automatically satisfied and moreover, if condition (\ref{T3}) holds, then condition (\ref{T2}) holds also. 
Condition \eqref{T3} can be written as \eqref{eq:suf-cond-net}.
\end{proof}

\subsection{Stability of the cycles}\label{subsec:stab-cycles}

Using the results from \cite{Podvigina_Ashwin_2011}, we can calculate the stability indices for points in all the connections in the network of the Jungle Game with four species which are contained in the coordinate planes. 
Therefore, when referring to the cycles in the Jungle Game network we study trajectories that occur in coordinate planes. This is a natural choice since the interactions are between two species at a time: when $S_4$ is eating $S_1$, the species $S_2$ and $S_3$ are not directly involved.

The next three propositions show that only the cycle among species $S_1$, $S_2$, and $S_4$ exhibits any form of stability, indeed the strongest type of stability admissible for a heteroclinic cycle in a network. Together with Theorem~\ref{th:stab-net}, we know that for almost any initial condition close to any part of the network its trajectory converges to the cycle $\Sigma_\text{142}$.

The stability of the heteroclinic cycles, or lack thereof, provides information about the survival of the species $S_i$, $i=1, \hdots, 4$.
Only the species corresponding to the nodes of a stable heteroclinic cycle are expected to coexist. Because the heteroclinic cycle with four nodes, $[\xi_1 \rightarrow \xi_4 \rightarrow \xi_3 \rightarrow \xi_2 \rightarrow \xi_1]$, is always unstable we know that the coexistence of all four species is not expected in the long-run. The same applies to the heteroclinic cycle with three nodes, $[\xi_1 \rightarrow \xi_4 \rightarrow \xi_3 \rightarrow \xi_1]$, meaning that these three species are not expected to coexist in the long-run. The essential asymptotic satbility of the heteroclinic cycle with three nodes, $[\xi_1 \rightarrow \xi_4 \rightarrow \xi_2 \rightarrow \xi_1]$, shows that the species $S_1, S_2, S_4$ coexist in the long-run while the species $S_3$ becomes extinct. Furthermore, the connections in this heteroclinic cycle indicate that the surviving species interact as a RSP game.

\begin{proposition} \label{prop:cycle142}
Assume that \eqref{eq:assumption} and \eqref{eq:suf-cond-net} hold.
Then the cycle $\Sigma_\text{142}$ is essentially asymptotically stable and the local stability indices are:

\begin{equation*} 
\sigma_{14} = 1-\frac{e_B}{e_A}>0, \qquad \quad \sigma_{42}= +\infty \qquad and \qquad \sigma_{21} = +\infty
\end{equation*}
\end{proposition}

\begin{proof}
 For this cycle, only for the Jacobian matrix at $\xi_4$ there is a positive transverse eigenvalue. Re-ordering the indices so that the quantity $b_1$ in Lemma~\ref{lemma1:calc-index} corresponds to $\xi_4$, we obtain
\begin{eqnarray*}
a_4a_2a_1&=&\dfrac{c_Ac_Bc_D}{e_Ae_Be_D}>1 \qquad \quad \; \;
\end{eqnarray*}
\vspace{-0.2cm}
and
\vspace{-0.2cm}
\begin{eqnarray*}
b_4a_2a_1+b_1a_2+b_2 &=& -\dfrac{e_Bc_Ac_B}{e_Ae_Be_D}+\dfrac{c_A^2}{e_Be_D}+\dfrac{c_B}{e_B}\\[0.1cm]
&=& \dfrac{c_Be_Ae_D+c_A(c_Ae_A-c_Be_B)}{e_Ae_Be_D} > 0.
\end{eqnarray*}
The first inequality follows from the first in \eqref{eq:assumption}.
The last inequality follows from the last two in \eqref{eq:assumption}.
Then, 
\begin{eqnarray*}
\sigma_{14}&=& f^\textnormal{index}(b_4,1)=f^\textnormal{index}\left(-\dfrac{e_B}{e_A},1\right)=\dfrac{e_A}{e_B}-1>0 \\[0.1cm]
\sigma_{42}&=& +\infty \\[0.1cm]
\sigma_{21}&=& f^\textnormal{index}(b_1+b_4a_1,1)=f^\textnormal{index}\left(\dfrac{c_Ae_A-c_Be_B}{e_Ae_D},1\right)=+\infty
\end{eqnarray*}
\end{proof}

We note that the connection $C_{42}$ is 2-dimensional. Proposition~\ref{prop:cycle142} calculates the stability index for the 1-dimensional trajectory $\kappa_{42} \subset C_{42}$ contained in the coordinate plane spanned by the axes containing $\xi_2$ and $\xi_4$. 

\begin{proposition} \label{prop:cycle143}
Assume that \eqref{eq:assumption} and \eqref{eq:suf-cond-net} hold.
Then the cycle $\Sigma_\text{143}$ is completly unstable.
\end{proposition}

\begin{proof}
This cycle satisfies the hypotheses of Lemma~\ref{lemma2:calc-index}, where two transverse eigenvalues are positive: at $\xi_3$ and at $\xi_4$. 
Denoting the quantities of subsection \ref{subsec:cycles} for this cycle by $\bar{a}_j$ and $\bar{b}_j$ and again re-ordering the indices we obtain, as  in Proposition~\ref{prop:cycle142},
\begin{eqnarray*}
\bar{a}_4\bar{a}_3\bar{a}_1&=&\dfrac{c_Ac_Bc_D}{e_Ae_Be_D}>1 \qquad \quad \; \;
\end{eqnarray*}
and
\begin{eqnarray*}
\bar{b}_3\bar{a}_4\bar{a}_1+\bar{b}_4\bar{a}_1+\bar{b}_1 &=& -\dfrac{e_Bc_Ac_D}{e_Ae_Be_D}-\dfrac{e_Ac_A}{e_Be_D}+\dfrac{c_B}{e_D} < 0 \Leftrightarrow\\[0.1cm]
&=& c_B < c_A\left(\dfrac{c_D}{e_A} + \dfrac{e_A}{e_B} \right).
\end{eqnarray*}
The last inequality holds from \eqref{eq:assumption} since $c_B<c_A$ and the term in brackets is greater than 2.
Then all $\sigma_j=-\infty$ and the cycle $\Sigma_{143}$ is not an attractor. 
\end{proof}

\begin{proposition} \label{prop:cycle1432}
Assume that \eqref{eq:assumption} and \eqref{eq:suf-cond-net} hold.
Then the cycle $\Sigma_\text{1432}$ is completely unstable.
\end{proposition}

\begin{proof}
In cycle $\Sigma_{142}$ the stability index for the connecting trajectory in $\Sigma_{142}$  from equilibrium $\xi_2$ to $\xi_1$ is $\sigma_{21}=+\infty$, according to Proposition~\ref{prop:cycle142}. Hence the stability index for the same connecting trajectory in any other cycle of the network is equal to $-\infty$. It follows from Corollary 4.1 in \cite{Podvigina_Ashwin_2011} that all stability indices are equal to $-\infty$ and the cycle is c.u.
\end{proof}

\subsection{Numerical results}\label{simulations}

In order to illustrate the results which we have analytically proved in Subsection~\ref{subsec:stab-cycles}, we present some images obtained by simulating the dynamics of \eqref{eq:4-species}.
We note that the same behaviour is observed for different parameter values as long as they satisfy the sufficient condition for the asymptotic stability of the network. We do not reproduce the figures here since they are identical to Figures~\ref{NumericResults} and \ref{NumericResultsZoom}.

Figure \ref{NumericResults} represents the time course of a trajectory of \eqref{eq:4-species} for some parameter values which fulfill the conditions for asymptotic stability of the network in Theorem~\ref{th:stab-net}. After some time the  trajectory approaches $\Sigma_{142}$, the e.a.s. cycle. As time increases the trajectory stays longer near each equilibrium.
The survival of the three species $S_1, S_2, S_4$ and their cyclic dominance can be seen by the successive periods of $x_i$ close to 1 for $i \neq 3$. The fact  that $x_3$ becomes equal to zero indicates the extinction of species $S_3$.
Note that this happens even though the initial condition in this example is close to the equilibrium $x_3=1$, $x_1=x_2=x_4=0$.

\begin{figure}[h!]
 \centering
\scalebox{0.53}{\includegraphics{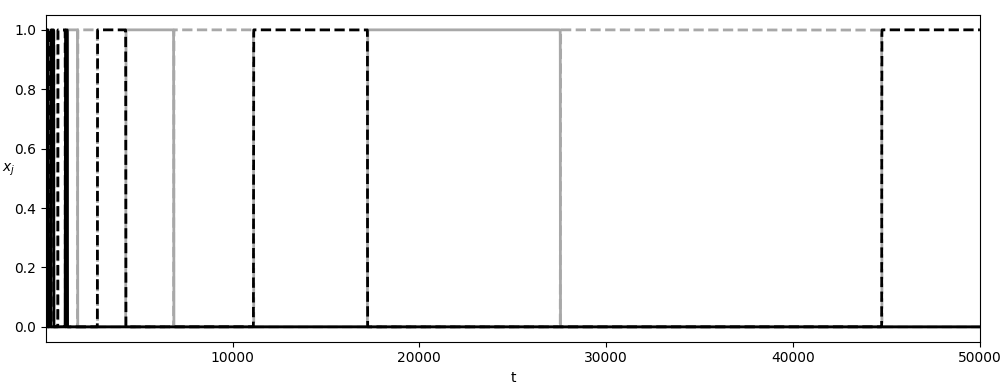}}
\caption{A typical time course for \eqref{eq:4-species}.
The lines in dashed grey, solid grey, solid black and dashed black are the coordinates $x_1, x_2, x_3$ and $x_4$ respectively. The parameter values for the simulation are: $c_A=1.2$, $c_B=1$, $c_D=1.1$, $e_A=0.7$, $e_B=0.65$, $e_D=0.72$. 
The initial condition is $x_1=x_2=x_4=0.1$ and $x_3=1$.}
\label{NumericResults}
\end{figure}

\begin{figure}[h!]
\centering
\scalebox{0.53}{\includegraphics{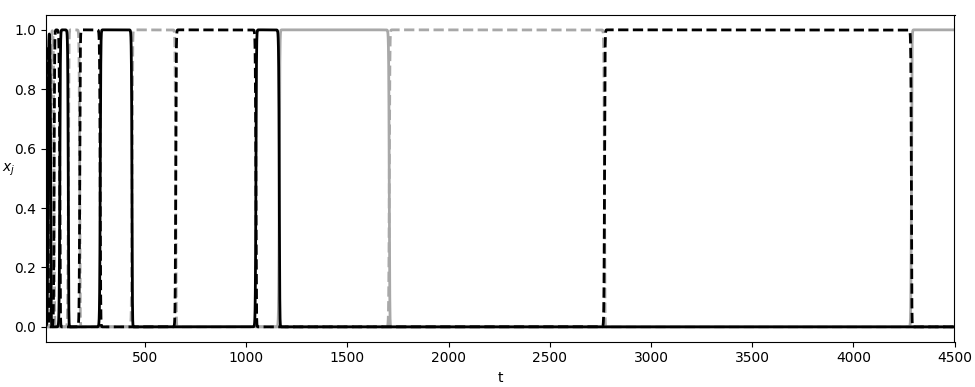}}
\caption{Detail of the trajectory of Figure~\ref{NumericResults} over a shorter time interval, showing a transient visit to $\xi_3$ (solid black). Conventions as in Figure~\ref{NumericResults}.}
\label{NumericResultsZoom}
\end{figure}

The same trajectory  as in Figure~\ref{NumericResults} is shown in a shorter time scale in Figure~\ref{NumericResultsZoom}. Before the trajectory reaches the e.a.s.\ cycle, some other sequences of equilibria are visited. The trajectory visits the cycle $\Sigma_{143}$ during some small period of time and once the trajectory leaves that cycle and visits the equilibrium $\xi_2$, then, for the remaining time it cyclically follows the sequence [$\xi_1 \to \xi_4 \to \xi_2 \to \xi_1$], i.e., the e.a.s.\ cycle.

\subsection{An alien in a game of three species}
In what remains of this section we interpret the Jungle Game with four species to obtain information about the outcome of having a fourth species trying to invade three species that play a RSP game. Let $S_1$, $S_2$ and $S_3$ be the original species playing RSP and assume that $S_{i-1}$ wins when confronted with $S_i$, $i \mbox{ mod } 3$. Let $A$ denote the invading species, an {\em alien}, which can be weak or strong depending on whether it wins when confronted with either one or two of the original species, respectively. We denote a weak alien by $A_w$ and a strong alien by $A_s$. The graphs of the interaction are represented in Figure~\ref{fig:graph-alien}. Without loss of generality, we choose $S_3$ as the only species that wins against $A_s$; we choose $S_3$ as the only species that loses against $A_w$. 

\begin{figure}[ht]
 \centering
\includegraphics[width=1.1\textwidth]{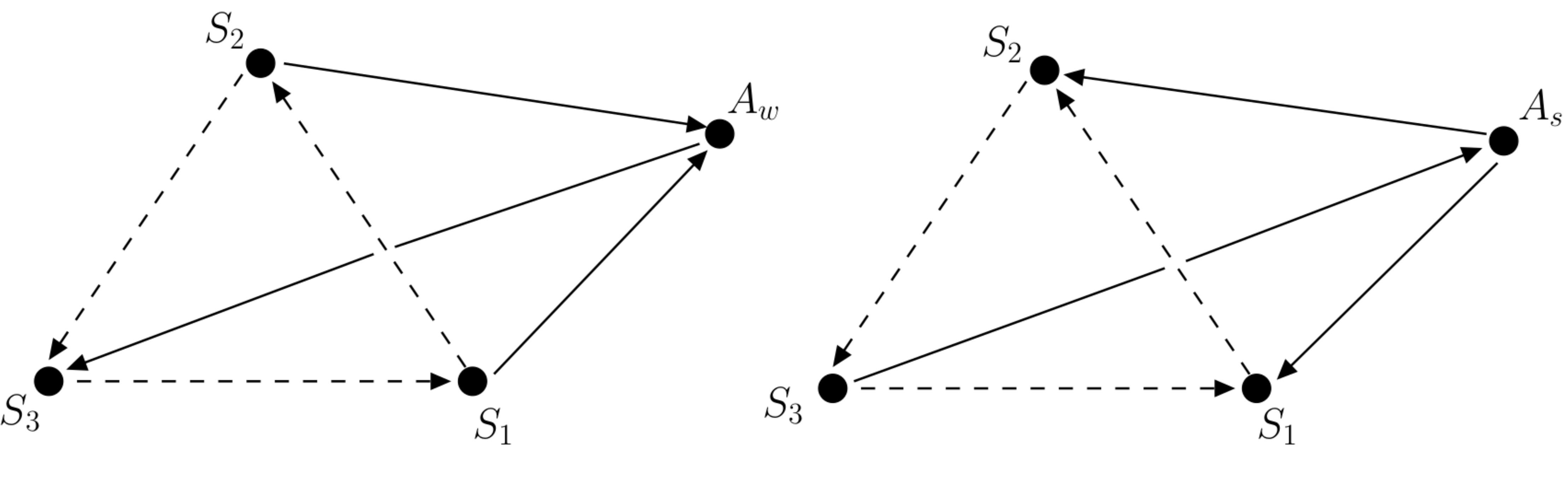}
\caption{Relationships of species in an invasion of a RSP game by an alien: the predator-prey relationships among the original population of RSP players and a weak alien, $A_w$ (left), and a strong alien, $A_s$ (right). The dashed lines indicate the interations among the original population of RSP players and the solid lines correspond to interactions with the alien.}
\label{fig:graph-alien}
\end{figure}

\begin{figure}[ht!]
\centering
{\includegraphics[width=0.7\textwidth]{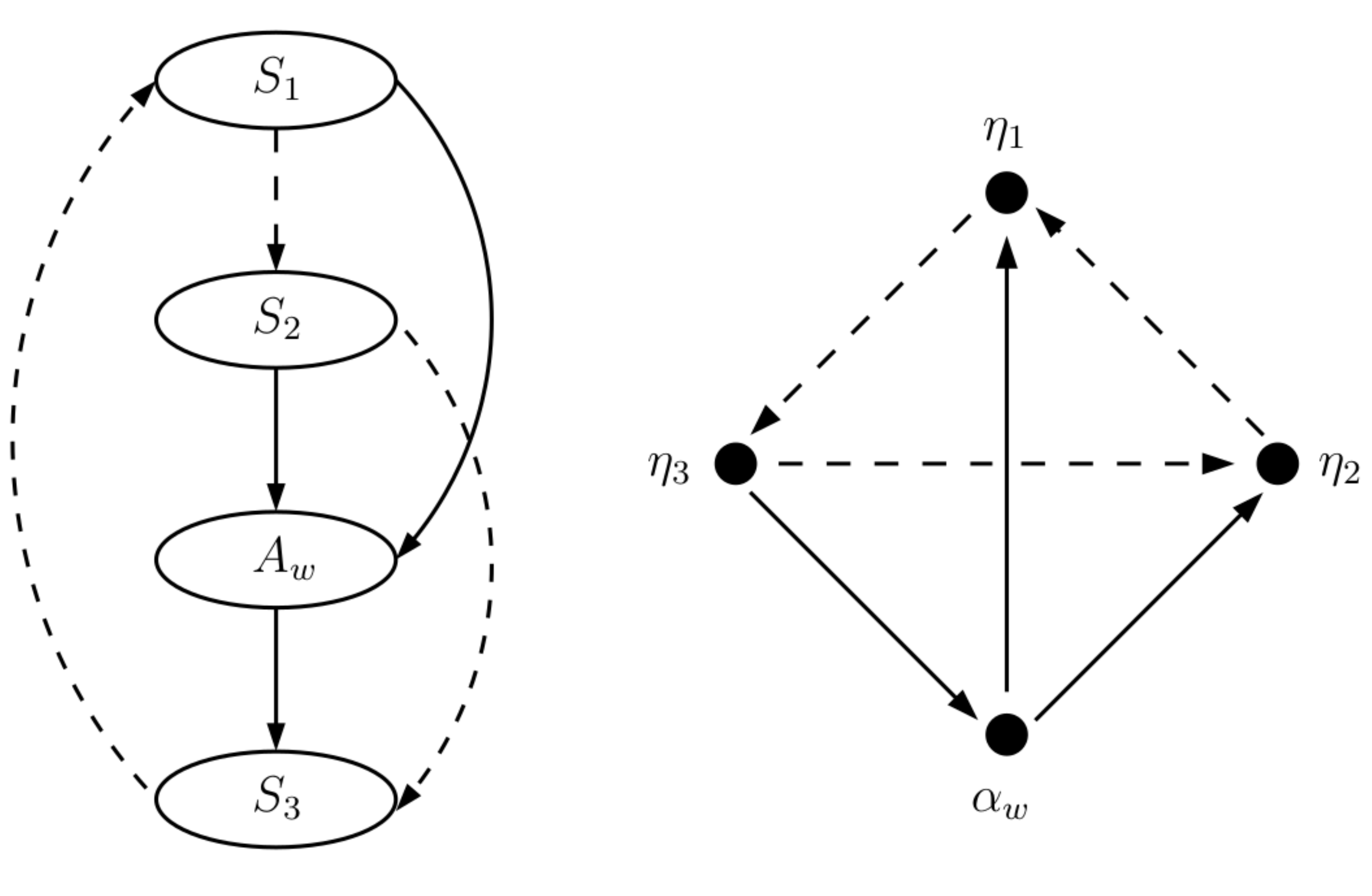}} \\
(a) The weak alien only wins against one species. \\
\mbox{} \\
{\includegraphics[width=0.7\textwidth]{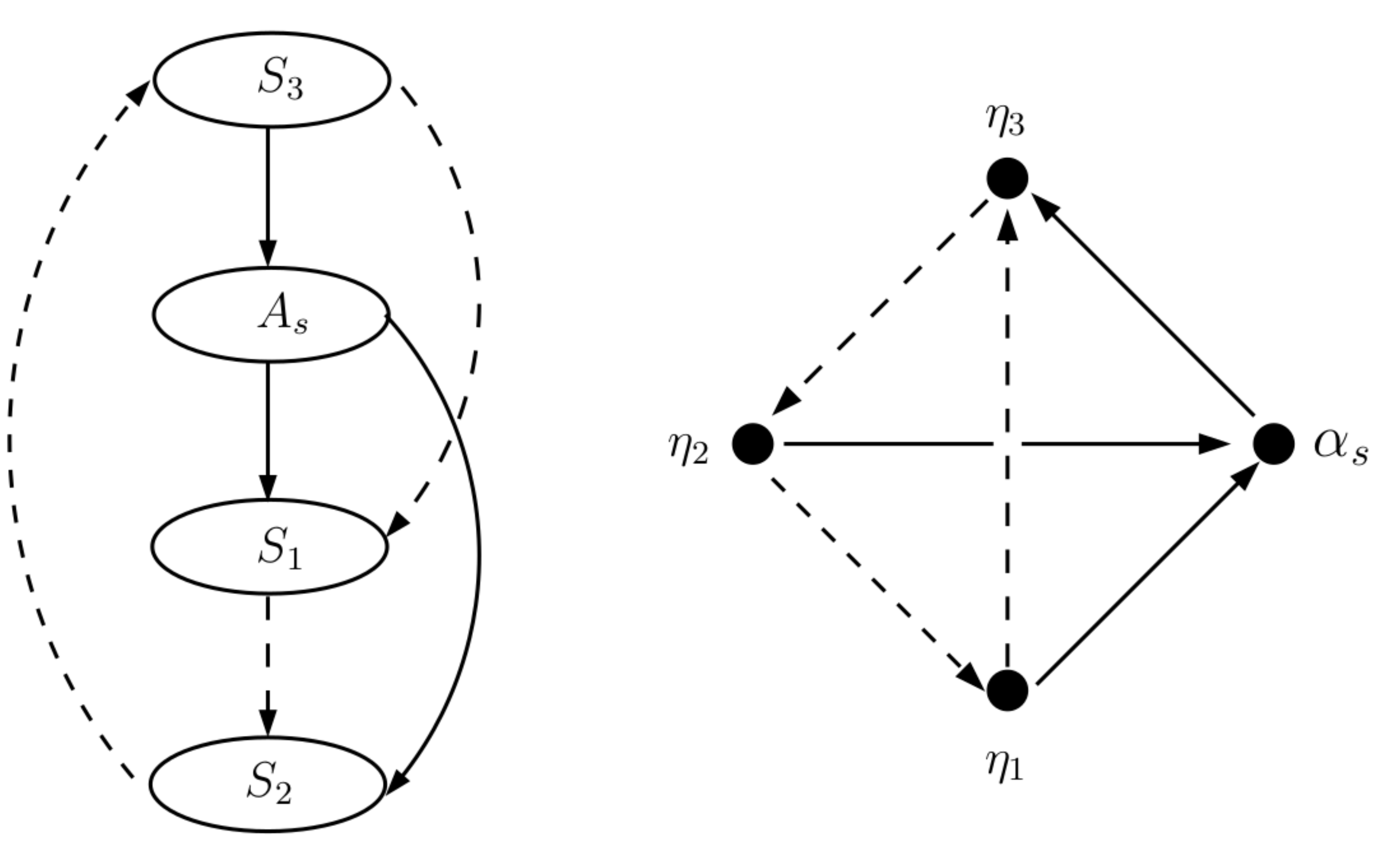}} \\
(b) The strong alien wins against two species.
\caption{On the left, the predator-prey relationships among the original population of RSP players and a weak alien in panel (a) and a strong alien in panel (b). The interactions are the same even though the species change place corresponding to different survival outcomes. On the right, the network of an invasion of a RSP game with nodes $\eta_j$ by a weak alien, $\alpha_w$, in panel (a) and a strong alien, $\alpha_s$, in panel (b). 
Dashed lines indicate the interations among the original population of RSP players and the solid lines correspond to interactions with the alien, as in Figure~\ref{fig:graph-alien}.}
\label{fig:graph_network-alien}
\end{figure}

It is simple to see that these correspond to the heteroclinic network of the Jungle Game with four species by comparing the relationships in Figure~\ref{fig:4-species}(a) with those in Figure~\ref{fig:graph_network-alien}. 
The results in Subsection~\ref{subsec:stab-cycles} establish that in the interaction with a weak alien the only e.a.s.\ cycle is that connecting the equilibria $\eta_1$, $\eta_2$, and $\eta_3$; when the fourth species is a strong alien the only e.a.s.\ cycle is that connecting the equilibria $\eta_2$, $\eta_3$, and $\alpha_s$.
Thus the weak alien $A_w$ is suppressed when interacting with a population of RSP players,
whereas the strong alien $A_s$ replaces one of the original species in playing a RSP game with the other two. It is noteworthy, although not surprising, that the original species which is replaced by $A_s$ is the species that loses when confronted with the species that beats the alien.

From Figure~\ref{fig:graph_network-alien} it is easy to see that after the introduction of a weak alien the original three species survive.
However, introducing a strong alien leads to the survival of the strong alien together with $S_2$ and $S_3$. 
This was anticipated by Case {\em et al.} \cite{Case2010} in their maxim ``the prey of the prey of the weakest is the least likely to survive'' albeit for a different model. From Figure~\ref{fig:graph_network-alien} (b) we see that $S_2$ is the weakest species and its prey is $S_3$. The prey of $S_3$ is $S_1$ which is the original species that becomes extinct by being replaced by the strong alien.
Of  $S_1$ and $S_2$ that are the two weakest, in the sense of losing to two others, the one who is also a prey of prey of one of the weakest does not survive.

\section{Concluding remarks}

We study the stability of the heteroclinic cycles of a Jungle Game with four species under Lotka-Volterra competition. We prove that only the cycle involving the bottom, top and second to the top species is stable while the whole network is asymptotically stable. Therefore, any initial configuration with four species evolves through competition to only the three species in the e.a.s.\ cycle connecting $S_1$, $S_2$, and $S_4$. These three species interact through a RSP game so that cyclic dominance of three species persists in the long-run.

The Jungle Game is re-interpreted as the invasion of a populatioin interacting under RSP to determine when the invading species replaces one of the original RSP players. This allows us to answer a question that differs from the usual ones about coexistence and extinction in that it is a question about {\em replacement} of species in the presence of an invader.
Moreover, we show that cyclic dominance  of four species is not a possible outcome in this setting.

\end{document}